\documentclass[12pt, reqno,draft]{amsart}
\usepackage{latexsym, amsmath, amssymb}
\usepackage{mathrsfs}
\usepackage{amsrefs}
\usepackage[usenames]{color}
\usepackage{graphicx}
\usepackage[all]{xy}
\usepackage{color,enumerate,cancel,hyperref}
\usepackage[utf8]{inputenc}

\newtheorem{Theorem}{Theorem}[section]
\newtheorem{Lemma}[Theorem]{Lemma}
\newtheorem{Proposition}[Theorem]{Proposition}
\newtheorem{Corollary}[Theorem]{Corollary}
 
\theoremstyle{remark}

\newtheorem{Definition}[Theorem]{Definition}

\newcommand{\Cont}{\ensuremath{\mathcal{C}}}
\newcommand{\LL}{\ensuremath{\mathcal{L}}}
\newcommand{\Id}{\ensuremath{\mathrm{Id}}}
\newcommand{\NNN}{\ensuremath{\mathbf{N}}}
\newcommand{\MM}{\ensuremath{\mathbf{M}}}
\newcommand{\ttt}{\ensuremath{\mathbf{t}}}
\newcommand{\sss}{\ensuremath{\mathbf{s}}}
\newcommand{\ww}{\ensuremath{\mathbf{w}}}
\newcommand{\ee}{\ensuremath{\mathbf{e}}}
\newcommand{\ZZ}{\ensuremath{\mathbb{Z}}}
\newcommand{\yy}{\ensuremath{\mathbf{y}}}
\newcommand{\YY}{\ensuremath{\mathbb{Y}}}
\newcommand{\xx}{\ensuremath{\mathbf{x}}}
\newcommand{\XX}{\ensuremath{\mathbb{X}}}
\newcommand{\NN}{\ensuremath{\mathbb{N}}}
\newcommand{\FF}{\ensuremath{\mathbb{F}}}
\newcommand{\UU}{\ensuremath{\mathcal{U}}}
\newcommand{\BB}{\ensuremath{\mathcal{B}}}
\newcommand{\WW}{\ensuremath{\mathcal{W}}}

\DeclareMathOperator{\wstar}{w*-}
\DeclareMathOperator{\supp}{supp}  

\numberwithin{equation}{section} 
\allowdisplaybreaks

\title[Primarity of direct sums of some sequence spaces]{Primarity of direct sums of Orlicz spaces and Marcinkiewicz  spaces}
\author[J. L. Ansorena]{Jos\'e L. Ansorena}
\address{Department of Mathematics and Computer Sciences\\
Universidad de La Rioja\\ 
Logro\~no\\
26004 Spain}
\email{joseluis.ansorena@unirioja.es}

\subjclass[2010]{46B25, 46B26, 46B15, 46B45}

\keywords{subsymmetric basis, primary Banach space, factorization of the identity,  Marcinkiewicz  space, Lorentz space, Orlicz  space, sequence space}

\begin{document}

\begin{abstract} 
Let $\YY$ be either an Orlicz sequence space or a Marcinkiewicz sequence space. We take advantage of the recent advances in the theory of factorization of the identity carried on by Lechner \cite{Lechner} to provide conditions on $\YY$ that  ensure that, for any $1\le p\le\infty$,  the infinite direct sum  of  $\YY$ in the sense of $\ell_p$
is  a primary Banach space, enlarging this way the list of  Banach spaces that are known to be primary.
\end{abstract}    
                    
\maketitle
\section{Introduction}
\noindent 
Within his study of operators through which the identity map factors, Lechner \cite{Lechner} introduced the following condition on the coordinate functionals of an unconditional basis of a Banach space.
\begin{Definition}\label{def:1}Let $(\xx_j)_{j=1}^\infty$ be an unconditional basis for a Banach space $\XX$. We say that its sequence $(\xx_j^*)_{j=1}^\infty$  of coordinate functionals verifies \textit{Lechner's condition} if for every $A\subseteq \NN$ infinite
and for every $\theta>0$ there is a sequence $(A_n)_{n=1}^\infty$ consisting of pairwise  disjoint infinite subsets of $A$ such that 
for every $(f_n^*)_{n=1}^\infty$ in $B_{\XX^*}$  there is a sequence of scalars $(a_n)_{n=1}^\infty\in S_{\ell_1}$ satisfying 
\[
\left\Vert \sum_{n=1}^\infty a_n \, P_{A_n}^*(f_n^*)\right\Vert \le \theta.
\]
\end{Definition}
Here, and throughout this note,  $B_\XX$ (respectively $S_\XX$) denotes the closed unit ball (resp. unit sphere) of a Banach space $\XX$. The symbol $P_A$ denotes the \textit{coordinate projection} on a set $A\subseteq\NN$ with respect to an  unconditional basis  $\BB=(\xx_j)_{j=1}^\infty$ of $\XX$, i.e., if $\BB^*=(\xx_j^*)_{j=1}^\infty$ is the sequence of coordinate functionals associated to the basis $\BB$,  also called the \textit{dual basic sequence} of $\BB$, then $P_A\colon\XX \to \XX$ is defined by
\begin{equation}\label{eq:intro1}
P_A(f)=\sum_{j\in A} \xx_j^*(f)\, \xx_j, \quad f\in \XX.
\end{equation}
Note that the \textit{dual coordinate projection} $P_A^*\colon \XX^* \to \XX^*$ of $P_A$ is given by
\[
P_A^*(f^*) =\wstar\sum_{j\in A} f^*(\xx_j)\, \xx_j^*, \quad f^*\in\XX^*.
\]  
Since the basis $\BB$ is, up to equivalence, univocally determined by the basic sequence $\BB^*$
(see \cite{AK}*{Corollary 3.2.4}) it is natural to consider Lechner's condition as a condition on  $\BB^*$ and $\XX^*$ instead of  a condition on $\BB$ and $\XX$.

In the aforementioned paper, Lechner achieved the following contribution to the theory of primary Banach spaces and factorization of the identity. Recall that a Banach space $\XX$ is said to be \textit{primary} if whenever $\YY$ and $\ZZ$ are Banach spaces such that $ \YY\oplus\ZZ \approx \XX$ then either $\YY\approx \XX$ or $\ZZ\approx \XX$. A basis is said to be \textit{subsymmetric} if it is unconditional and equivalent to all its subsequences. The infinite direct sum of a Banach space $\XX$
in the sense of $\ell_p$ (respectively $c_0$) 
will be denoted by $\ell_p(\XX)$ (resp. $c_0(\XX)$). $\LL(\XX)$ will denote the Banach algebra of automorphisms of a Banach space $\XX$.
We say that the identity map on  $\XX$  factors through an operator $R\in\LL(\XX)$ if there are operators $S$ and $T\in\LL(\XX)$ such that $T\circ R \circ S =\Id_\XX$. 

\begin{Theorem}[see \cite{Lechner}*{Theorems 1.1 and 1.2}]\label{thm:Lechner} Assume that $\XX$ is a Banach space provided with a subsymmetric basis whose dual basic sequence verifies Lechner's condition. Let $1\le p\le \infty$ and  let $\YY$ be either $\XX^*$ or  $\ell_p(\XX^*)$.
Then, given $T\in\LL(\YY)$, the identity map on $\YY$ factors through either $T$ or $\Id_\YY-T$. Consequently, $\ell_p(\XX^*)$  is a primary Banach space.
\end{Theorem}
Before undertaking the task of using Theorem~\ref{thm:Lechner} for obtaining new primary Banach spaces, we must go over the state of the art on this topic.
Casazza et al. \cite{CKL} proved that if $\XX$ has a \textit{symmetric} basis, i.e., a basis that is equivalent to all its permutations, then $c_0(\XX)$ and, if $1< p < \infty$ 
and $\XX$ is not isomorphic to $\ell_1$,  $\ell_p(\XX)$  are primary Banach spaces. Shortly later, Samuel \cite{Samuel} proved that $\ell_p(\ell_r)$  $c_0(\ell_r)$ and $\ell_r(c_0)$ are, for $1\le p,r<\infty$, primary Banach spaces.
Subsequently, Capon \cite{Capon} completed the study by proving that $\ell_1(\XX)$ and $\ell_\infty(\XX)$ are primary Banach spaces whenever $\XX$ possesses a symmetric basis. Symmetric bases are subsymmetric \cites{KP,Singer2}, and, in practice, the only information that one needs about symmetric bases in many situations is its subsymmetry. So, it is natural to wonder if the proofs carried on in  \cite{CKL} and \cite{Capon} still work when dealing with subsymmetric bases. A careful look at these papers reveals that it is the case. Summarizing, we have the following result.
\begin{Theorem}[see \cites{CKL,Samuel,Capon}]\label{thm:CKLSC} Let $\XX$ be a Banach space provided with a subsymmetric basis. Then $c_0(\XX)$ and 
$\ell_p(\XX)$, $1\le p \le \infty$, are primary Banach spaces.
\end{Theorem}
At this point, we must mention that, since Pe\l czy\'nski decomposition method is a key tool to face the study of primary Banach spaces,  proving that $\ell_p(\XX)$ is primary  is, in some sense, easier than proving that $\XX$ is. In fact, as far as we know, $\ell_p$, $1\le p<\infty$, and $c_0$ are  the only known primary Banach spaces provided with a subsymmetric basis.

In light of Theorem~\ref{thm:CKLSC},  applying Theorem~\ref{thm:Lechner} to a Banach space $\XX$ provided with a shrinking (subsymmetric) basis does not a add a new space to the list of primary Banach spaces. So, with the aim of finding new primary Banach spaces, taking into account \cite{AK}*{Theorem 3.3.1}, we must apply Theorem~\ref{thm:Lechner} to Banach spaces  $\XX$ containing a complemented copy of $\ell_1$. Among them, $\XX=\ell_1$ seems to be the first space we have to consider.
It is timely to bring up the following result.
\begin{Theorem}[see \cite{CKL}]\label{thm:lploo} Let $1\le p\le \infty$. Then $\ell_p(\ell_\infty)$ is primary.
\end{Theorem}
It is known \cite{Lechner}*{Remark 2.2} that the unit vector system of $\ell_\infty$, which is, under the natural pairing, the dual basic sequence of the unit vector system of $\ell_1$, satisfies Lechner's condition. This result, combined with Theorem~\ref{thm:Lechner}, provides an alternative proof to Theorem~\ref{thm:lploo}. From an opposite perspective, in order to take advantage of Theorem~\ref{thm:lploo}  for obtaining new primary Banach spaces, we have to find dual basic sequences, other than the unit vector system of  $\ell_\infty$, that are not boundedly complete and satisfy Lechner's condition. In this manuscript, we exhibit that the unit vector system of some classical sequence spaces fulfils these requirements.

The Banach spaces we deal with are Orlicz sequence spaces and Marcinkiewicz sequence spaces.
In Section~\ref{Orlicz}  we characterize, in terms of the convex Orlicz function $M$, when the unit vector system of the Orlicz sequence space $\ell_M$ safisfies Lechner's  condition.  In turn, in Section~\ref{Marcin} we describe those weights $\sss$ for which the  the unit vector system of the Marcinkiewicz space $m(\sss)$ satisfies Lechner's condition.
Previously to these sections, in Section~\ref{LechnerCondition}, we carry on a detailed analysis of Lechner's condition.

Throughout this article we follow standard Banach space terminology and notation as can be found in \cite{AK}. We single out the notation that  is more commonly employed. We will  denote by $\FF$ the real or complex field.  By a \textit{sign} we mean a  scalar of modulus one.  We denote by $(\ee_k)_{k=1}^\infty$  the unit vector system of $\FF^\NN$, i.e.,  $\ee_k=(\delta_{k,n})_{n=1}^\infty$, were $\delta_{k,n}=1$ if $n=k$ and $\delta_{k,n}=0$ otherwise. The linear span of the unit vector system will be denoted by $c_{00}$.

Given families of non-negative real numbers $(\alpha_i)_{i\in I}$ and $(\beta_i)_{i\in I}$ and a constant $C<\infty$, the symbol $\alpha_i\lesssim_C \beta_i$ for $i\in I$  means that $\alpha_i\le C \beta_i$ for every $i\in I$,
while $\alpha_i\approx_C \beta_i$ for $i\in I$ means that $\alpha_i\lesssim_C \beta_i$ and $\beta_i\lesssim_C \alpha_i$ for $i\in I$.  
A \textit{basis} will be a Schauder basis. Suppose $(\xx_j)_{j=1}^\infty$ and $(\yy_j)_{j=1}^\infty$ are bases. We say that $(\yy_j)_{j=1}^\infty$ $C$-\textit{dominates} $(\xx_j)_{j=1}^\infty$ (respectively  $(\yy_j)_{j=1}^\infty$ is $C$-\textit{equivalent} to $(\xx_j)_{j=1}^\infty$), 
and write $(\xx_j)_{j=1}^\infty\lesssim_C(\yy_j)_{j=1}^\infty$  (resp. $(\xx_j)_{j=1}^\infty\approx_C(\yy_j)_{j=1}^\infty$)  if 
\[
\textstyle
\Vert\sum_{j=1}^\infty a_j\xx_j\Vert\lesssim_C\Vert\sum_{j=1}^\infty a_j\yy_j\Vert \text{ (resp. }
\Vert\sum_{j=1}^\infty a_j\xx_j\Vert\approx_C\Vert\sum_{j=1}^\infty a_j\yy_j\Vert)
\]
for $(a_j)_{j=1}^\infty\in c_{00}$. In all the above cases, when the value of the constant $C$ is irrelevant, we simply drop it from the notation.  A basis is said to be \textit{unconditional} if all its permutations are basic sequences. I turn, we say that  a basis $(\xx_j)_{j=1}^\infty$ is $C$-unconditional if $(\xx_j)_{n=1}^\infty\approx_C  (\epsilon_j\xx_j)_{j=1}^\infty$ for any choice of signs $(\epsilon_j)_{j=1}^\infty$. If $(\xx_j)_{j=1}^\infty$ is a $C$-unconditional basis of a Banach space $\XX$ and $A\subseteq\NN$, then the operator $P_A$ defined as in \eqref{eq:intro1} is well-defined and satisfies $\Vert P_A \Vert \le C$. It is well-known (see e.g. \cite{AK}*{Proposition 3.1.3}) that a basis $\BB$ is unconditional if and only if there exists a constant $C\ge 1$ such that $\BB$ is $C$-unconditional.

We say that a sequence in a Banach space is a \textit{basic sequence} if it is a  basis of its closed linear span. If $\BB$ is a basis of a Banach space $\XX$, then its coordinate functionals constitute a basic sequence in $\XX^*$. Reciprocally, if $\BB$ is a basis of a Banach space $\YY$ and $\XX:=\XX[\BB]$ is the closed linear span of $\BB^*$ in $\YY^*$, then there is a natural isomorphic embedding  of $\YY$ into  $\XX^*:=\YY[\BB]$ and, via this embedding, $\BB$ is the dual basic sequence of $\BB^*$  (see \cite{AK}*{Proposition 3.2.3 and Corollary 3.2.4}). Consequently, any basic sequence is the dual basic sequence of some basis. So, it makes sense to wonder if a given unconditional basic sequence in a Banach space  (regarded as a sequence in the Banach space $\YY[\BB]$ constructed as above)  satisfies Lechner's condition.

A basis $\BB=(\xx_j)_{j=1}^\infty$ of $\XX$ is said to be \textit{boundedly complete} if whenever $(a_j)_{j=1}^\infty\in\FF^\NN$ satisfies $\sup_n\Vert \sum_{j=1}^n a_j \, \xx_j\Vert<\infty$ there is $f\in \XX$ such that $\xx_j^*(f)=a_j$ for every $j\in\NN$. The basis $\BB$ is said to be \textit{shrinking} if $\BB^*$ is a basis of the whole space $\XX^*$. It is known \cite{James1950} that a basis $\BB$ is boundedly complete if and only if $\BB^*$ is shrinking.

The \textit{support} of a vector $f\in\XX$ with respect to the basis $\BB$ is the set 
\[
\supp(f)=\{j\in\NN \colon \xx_j^*(f)\not=0\},
\]
and the  \textit{support} of a functional $f^*\in\XX^*$  with respect to the basis $\BB$ is the set 
\[
\supp(f^*)=\{j\in\NN \colon f^*(\xx_j)\not=0\}.
\]
A sequence $(f_n)_{n=1}^\infty$ in either $\XX$ or $\XX^*$ is said to be \textit{disjointly supported} if $(\supp(f_n))_{n=1}^\infty$ is a sequence of pairwise disjoint subsets of $\NN$. A \textit{block basic sequence} is a sequence $(f_n)_{n=1}^\infty$ for which there is an increasing sequence $(k_n)_{n=1}^\infty$ of positive integers such that, with the convention $n_0=0$, $\supp(f_n)\subseteq[1+k_{n-1},k_n]$ for every $n\in\NN$. Block basic sequences are an special case of disjointly supported sequences. Since any block basic sequence is a basic sequence, our terminology is consistent.  Note that any disjointly supported sequence (in either $\XX$ or $\XX^*$) with respect to an unconditional basis of a Banach space $\XX$ is an unconditional basic sequence.

Let $\XX\subseteq\FF^\NN$ be a Banach space for which  the unit vector system is a basis. We say that a Banach space $\YY\subseteq\FF^\NN$ is the \textit{dual space of $\XX$ under the natural pairing} if there is an isomorphism $T\colon\XX^*\to\YY$ such that $T(f)(g)=\sum_{j=1}^\infty a_j b_j$ for every $f=(a_j)_{j=1}^\infty\in\YY$ and every $g=(b_j)_{j=1}^\infty\in c_{00}$.  Observe that if $\YY$ is, under the natural pairing, the dual space of $\XX$ and $f=(a_j)_{j=1}^\infty$ belongs to either $\XX$ or $\YY$, then the support of $f$ with respect to the unit vector system is the set $\supp(f)=\{j \in\NN \colon a_j\not=0\}$.

A sequence $=(f_j)_{j=1}^\infty$ is a Banach space is said to be \textit{semi-normalized} if  $\inf_j\Vert f_j \Vert>0$ and $\sup_j \Vert f_j \Vert <\infty$. Note that subsymmetric bases are semi-normalized.

The symbol $f=\wstar\sum_{n=1}^\infty f_n$ means that the series $\sum_{n=1}^\infty f_n$ in $\XX^*$ converges to $f\in\XX^*$ in the weak* topology of the dual space $\XX^*$.
Recall that if $(\xx_j)_{j=1}^\infty$ is a basis of a Banach space $\XX$ and $(\xx_j^*)_{j=1}^\infty$ is its sequence of coordinate functionals, then, for every $f^*\in\XX^*$, 
$(f^*(\xx_j))_{j=1}^\infty$ is the unique sequence $(a_j)_{j=1}^\infty\in\FF^\NN$ such that
 $f^*=\wstar\sum_{j=1}^\infty a_j \, \xx^*_j$.
 
Other more specific notation  will be specified in context when needed.

\section{Lechner's condition}\label{LechnerCondition}
\noindent
The main goal of the study carried on in this section is to show that Lechner's condition has a simpler form when the target basis is subsymmetric. In order to prove our results, it will be convenient to introduce some notation. 

If $\BB=(\xx_j)_{j=1}^\infty$ is a subsymmetric basis of a Banach space $\XX$ then \cite{Ansorena}*{Theorem 3.7} there is a renorming of $\XX$ with respect to which it is $1$-subsymmetric, i.e., $\BB$ is $1$-unconditional and for every increasing map $\phi\colon\NN\to\NN$ the linear operator
\begin{equation}\label{equation:vf}
V_\phi\colon\XX\to \XX,\quad  \sum_{j=1}^\infty a_j \, \xx_j \mapsto  \sum_{j=1}^\infty a_j \, \xx_{\phi(j)}
\end{equation}
is an isometric embedding. If the basis is  $1$-subsymmetric then the linear operator 
\begin{equation}\label{equation:uf}
U_\phi\colon\XX\to \XX,\quad  \sum_{j=1}^\infty a_j \, \xx_j \mapsto  \sum_{j=1}^\infty a_{\phi(j)}\, \xx_{j}
\end{equation}
is norm-one for every increasing map $\phi\colon\NN\to\NN$  (see e.g.  \cite{Ansorena}*{Lemma 3.3}). The dual operators of  $U_\phi$ and $V_\phi$ are given by  
\begin{align*}
V_\phi^*&\colon\XX^*\to \XX^*,\quad  \wstar\sum_{j=1}^\infty a_j \, \xx_j^* \mapsto   \wstar\sum_{j=1}^\infty a_{\phi(j)}\, \xx_{j}^*,\\
U_\phi^*&\colon\XX^*\to \XX^*,\quad  \wstar\sum_{j=1}^\infty a_j \, \xx_j^* \mapsto   \wstar\sum_{j=1}^\infty a_j \, \xx_{\phi(j)}^*.
\end{align*}
Since $U_\phi\circ V_\phi=\Id_\XX$, we have $V_\phi^*\circ U_\phi^*=\Id_{\XX^*}$. Consequently, $U_\phi^*$ is an isomorphic embedding
(isometric embedding if $\BB$ is $1$-subsymmetric). 

For reference, we write down the following elemental lemma, which will be used in the subsequent theorem.
\begin{Lemma}\label{lem:lechner:9}Let $(B_n)_{n=1}^\infty$ be a sequence of disjointly supported subsets of $\NN$ and $(A_n)_{n=1}^\infty$ be a sequence of 
disjointly supported infinite subsets of $\NN$. There exists an increasing map $\phi\colon\NN\to\NN$ such that $\phi(B_n)\subseteq A_n$ for every $n\in\NN$.
\end{Lemma}

\begin{proof}Clearly, it suffices to prove the result in the case when  $(B_n)_{n=1}^\infty$ is a partition of $\NN$. Define $\nu\colon\NN\to \NN$ by $\nu(k)=n$ if $k\in B_n$.
By hypothesis, 
\[
D_{n,m}:=\{ j \in A_n \colon j>m\}
\]
is non-empty for every $n\in\NN$ and $m\in\NN\cup\{0\}$. With the convention $\phi(0)=0$, we define $\phi\colon\NN\to\NN$ by mean of the recursive formula
\[
\phi(k)=\min D_{\nu(k),\phi(k-1)}, \quad k\in\NN.
\]
Its clear  that $\phi$ satisfies the desired properties.
\end{proof}

We are almost ready to prove the aforementioned characterization of Lechner's condition under the assumption that the target basis is subsymmetric.
Before doing so, we bring up a result that is implicit in \cite{Lechner}.
\begin{Lemma}[cf. \cite{Lechner}*{Remark 2.1}]\label{lem:lechner:13}Let $\BB$ be an unconditional basis of a Banach space. Assume that $\BB^*$ fails to verify Lechner's condition. Then there is a sequence of disjointly supported functionals in $\XX^*$  that is equivalent  to the unit vector system of $\ell_1$. 
\end{Lemma}
\begin{proof}Our hypothesis says that there are $\theta>0$ and $A\subseteq\NN$ infinite such that for every sequence $(A_n)_{n=1}^\infty$ consisting of  pairwise disjoint infinite subsets of $A$ there is $(f_n^*)_{n=1}^\infty$ in $B_{\XX^*}$ such that
$
\theta<\Vert \sum_{n=1}^\infty a_n \, P_{A_n}^*(f_n^*)\Vert
$
for every $(a_n)_{n=1}^\infty\in S_{\ell_1}$. 

Pick out an arbitrary sequence  $(A_n)_{n=1}^\infty$ consisting  of  pairwise disjoint infinite subsets of $A$ and let $(f_n^*)_{n=1}^\infty$  be as above.
If $g_n^*= P_{A_n}^*(f_n^*)$, we have $\supp(g_n^*)\subseteq A_n$ and $\sup_n \Vert g_n^*\Vert <\infty$ for every $n\in\NN$. We infer that the disjointly supported sequence $(g_n^*)_{n=1}^\infty$ is equivalent to the unit vector basis of $\ell_1$.
\end{proof}

\begin{Theorem}\label{thm:lechner:2}Assume that $\BB$  is a subsymmetric basis for a Banach space $\XX$. Then  its dual basic sequence $\BB^*$ fails to verify Lechner's condition  if and only  there is a sequence of disjointly supported functionals in $\XX^*$  that is equivalent  to the unit vector system of $\ell_1$. 
\end{Theorem}
\begin{proof}The ``only if''  follows from Lemma~\ref{lem:lechner:13}.
Assume that  there is a disjointly supported sequence  $(f_n^*)_{n=1}^\infty$  in $\XX^*$ that is equivalent to the unit vector system of $\ell_1$. 
By dilation, we can assume that $\Vert f_n^*\Vert \le 1$ for every $n\in\NN$. Let $c>0$ be such that
\[ 
c\sum_{n=1}^\infty |a_n| \le \left\Vert \sum_{n=1}^\infty a_n \, f_n \right\Vert, \quad (a_n)_{n=1}^\infty\in\ell_1.
\]
We also assume, without loss of generality, that $\BB$  is $1$-subsym\-metric.
Choose $0<\theta<c$ and $A=\NN$. Pick a sequence $(A_n)_{n=1}^\infty$ consisting of pairwise disjoint subsets of $\NN$. By Lemma~\ref{lem:lechner:9}, there is an increasing map 
$\phi\colon\NN\to\NN$ such that
$\phi(\supp(f_n^*))\subseteq A_n$.  Put  $g_n^*=U_\phi^*(f_n^*)$ for $n\in\NN$.
Then, taking into account that $U_\phi^*$ is an isometric embedding,  we have $P_{A_n}^*(g_n^*)=g_n^*\in B_{\XX^*}$ for every $n\in\NN$, and
\[
\theta<c\le \left\Vert \sum_{n=1}^\infty a_n \, g_n^* \right\Vert
\]
for every $(a_n)_{n=1}^\infty \in S_{\ell_1}$. Consequently, $\BB$ does not satisfy Lechner's condition.\end{proof}

Note that, if $\XX$ has an unconditional basis and $\XX^*$ is non-separable, then 
$\ell_1$ is a subspace of $\XX^*$ (see \cite{AK}*{Theorems 2.5.7 and 3.3.1}). So, Theorems~\ref{thm:Lechner}  and~\ref{thm:lechner:2} reveal that the position in which $\ell_1$ is (and is not) placed inside 
 $\XX^*$ has significative structural consequences.
 
Next, we give some consequences of Theorem~\ref{thm:lechner:2}. First of them was previously achieved by Lechner \cite{Lechner}.

\begin{Lemma}\label{lem:Lechner:11}Let  $\BB=(f_n)_{n=1}^\infty$ be  a semi-normalized disjointly supported sequence in $\ell_\infty$. Then $\BB$ is equivalent to the
unit vector system of $\ell_\infty$.
\end{Lemma}
\begin{proof}Denote  $c=\inf_n\Vert f_n \Vert$ and $C=\sup_n \Vert f_n \Vert$.
It is clear that $c \Vert g \Vert_\infty \le \Vert \sum_{n=1}^\infty a_n f_n\Vert \le C \Vert g \Vert_\infty$ for every $g=(a_n)_{n=1}^\infty\in c_{00}$.
\end{proof}

\begin{Proposition}[see \cite{Lechner}*{Remark 2.2}]\label{prop:lechner:3}The unit vector system of $\ell_\infty$ satisfies Lechner's condition. \end{Proposition}
\begin{proof}The unit vector system of $\ell_\infty$,  denoted  by $\BB_\infty$ in this proof, is, under the natural pairing, the dual basic sequence of the unit vector basis of $\ell_1$, denoted by $\BB_1$. Assume that there is a disjointly supported sequence $\BB$ in $\ell_\infty$ that is equivalent to $\BB_1$. Then, in particular, $\BB$ is semi-normalized.
Invoking Lemma~\ref{lem:Lechner:11} we obtain $\BB_1\approx \BB \approx \BB_\infty$. This absurdity, combined with Theorem~\ref{thm:lechner:2}, proves that $\BB_\infty$ fulfils Lechner's condition.
\end{proof}

\begin{Proposition}\label{prop:lechner:4}Let $\BB$ be a subsymmetric basis of a Banach space $\XX$ whose dual basic sequence $\BB^*$ satisfies Lechner's condition. Then $\BB$ is boundedly complete and $\BB^*$ is shrinking.
\end{Proposition}

\begin{proof}If $\BB^*$ fails to be  shrinking, then, by \cite{AK}*{Theorem 3.3.1}, there is a block basis with respect to $\BB^*$ that is equivalent to the unit vector system of $\ell_1$. Consequently, by  Theorem~\ref{thm:lechner:2}, $\BB^*$ does not satisfy Lechner's condition. We close the proof by invoking \cite{AK}*{Theorem 3.2.15}. \end{proof}

\begin{Corollary}\label{cor:lechner:5}Let $\XX$ be a Banach space provided with a subsymmetric  basis $\BB$. If  $\BB^*$ satisfies Lechner's condition, then $\XX$ is a dual space (and $\XX^*$ is a bidual space).
\end{Corollary}
\begin{proof}It is immediate from combining Proposition~\ref{prop:lechner:4} with \cite{AK}*{Theorem 3.2.15}.
\end{proof}

\begin{Corollary}\label{cor:lechner:6}Let $\XX$ be a Banach space provided with a subsymmetric shrinking basis $\BB$. If  $\BB^*$ satisfies Lechner's condition, then $\XX$ is reflexive.
\end{Corollary}
\begin{proof}Just combine Proposition~\ref{prop:lechner:4} with \cite{AK}*{Theorem 3.2.19}.\end{proof}

\begin{Proposition}\label{prop:lechner:7}Let $\XX$ be a reflexive Banach space provided with a subsymmetric basis $\BB$. Then $\BB^*$ satisfies Lechner's condition.\end{Proposition}
\begin{proof} By \cite{AK}*{Theorems 3.2.15, 3.2.19 and 3.3.1},  $\XX^*$  contains no subspace isomorphic to $\ell_1$.  Then, by  Theorem~\ref{thm:lechner:2}, $\BB^*$ satisfies Lechner's condition.
\end{proof}

\section{Marcinkiewicz spaces}\label{Marcin}
\noindent
A weight will be a sequence of positive scalars. Given a weight $\sss=(s_j)_{j=1}^\infty$ the \textit{Marcinkiewicz space} $m(\sss)$ is the set of all sequences $f=(a_j)_{j=1}^\infty\in\FF^\NN$ such that
\[
\Vert f\Vert_{m(\sss)}:=
\sup\left\lbrace \frac{1}{s_n} \sum_{j\in A} |a_j| \colon  n\in\NN, \, |A|=n\right\rbrace<\infty.
\]
It is clear, and well-known, that $(m(\sss), \Vert \cdot \Vert_{m(\sss)})$ is a Banach space and that the unit vector system is a symmetric basic sequence in $m(\sss)$.
If $f\in c_0$ and $(a_n^*)_{n=1}^\infty$ denotes its non-increasing rearrangement then 
\begin{equation}\label{eq:MarNorm}
\Vert f\Vert_{m(\sss)}=\sup_n \frac{1}{s_n} \sum_{k=1}^n a_k^*.
\end{equation}
It is not hard to prove that if
$
\lim_n {s_n}/{n}=0
$
then $m(\sss)\subseteq c_0$ continuously. Otherwise, we have $m(\sss)=\ell_\infty$ (up to an equivalent norm).

Next proposition gathers some results relating Marcinkiewicz spaces to Lorentz spaces. Prior to enunciate it, let us fix some terminology. $\WW$ will denote the set consisting of all non-increasing weights $(w_j)_{j=1}^\infty$ such that $w_1=1$ and $\lim_j w_j=0$. A weight $(w_j)_{j=1}^\infty$ is said to be \textit{regular} if
\[
\sup_n \frac{1}{n w_n} \sum_{j=1}^n w_j <\infty.
\]
If two weights  $\ww=(w_j)_{j=1}^\infty$ and $\sss=(s_n)_{n=1}^\infty$ are related by the formula
\[
s_n=\sum_{j=1}^n w_j, \quad n\in\NN
\]
we say that $\sss$ is the \textit{primitive weight} of $\ww$ and that $\ww$ is the \textit{discrete derivative} of $\sss$.
Given a weight $\ww=(w_j)_{j=1}^\infty\in \WW$ and $1\le p\le\infty$ the  Lorentz space $d(\ww,p)$ is the set  of all sequences $f \in c_0$ whose non-increasing rearrangement $(a_j^*)_{j=1}^\infty$ fulfils
\[
\Vert f \Vert_{d(\ww,p)}:=\left\Vert (a_j^* w_j)_{j=1}^\infty\right\Vert_p<\infty.
\]
Let $\sss=(s_n)_{n=1}^\infty$ be the primitive weight of $\ww$, Since $\sss$ is doubling, $d(\ww,p)$ is a quasi-Banach space (see \cite{CRS}*{Theorems 2.2.13,  2.2.16 and  2.3.1}). Moreover, since $(s_n/n)_{n=1}^\infty$  is non-increasing, $d(\ww,1)$ is a Banach space (see \cite{CRS}*{Theorem  2.5.10}). It is not hard to prove that $c_{00}$ is a dense subspace of $d(\ww,1)$. Then,  the unit vector system is a symmetric basis of $d(\ww,1)$.
\begin{Proposition}[See \cite{CRS}*{Theorems 2.4.14 and 2.5.10 and  Corollary 2.4.26}; see also \cite{AADK}*{Section 6}]\label{prop:Mar1}
Let $\ww=(w_j)_{j=1}^\infty\in\WW$, let $\sss$ denote its primitive weight, and let $\ww^{-1}=(1/w_j)_{j=1}^\infty$ be its inverse weight.
\begin{itemize}
\item[(i)] $m(\sss)$ is, under the natural pairing, the dual space of $d(\ww,1)$.
\item[(ii)]  If $\ww$ is a regular weight, then $d(\ww^{-1},\infty)=m(\sss)$ (up to an equivalent quasi-norm). 
\end{itemize}
\end{Proposition}

In route to state the main result of the section, we introduce some additional conditions on weights, and we  bring up a result involving them. We say that a weight $(s_n)_{n=1}^\infty$ is \textit{essentially decreasing}  (respectively essentially increasing) if   
\[
\inf_{m\le n} \frac{s_m}{s_n}>0 \text{ (resp. }
\sup_{m\le n} \frac{s_m}{s_n}<\infty). 
\]
Note that $(s_n)_{n=1}^\infty$ is essentially decreasing  (resp. essentially increasing) if and only if it is equivalent to a  non-increasing (resp.  non-decreasing)  weight. 
We say that a weight $(s_n)_{n=1}^\infty$ has the \textit{lower regularity property} (LRP for short) if  there is  a constant $C>1$ and an integer $r\ge 2$ such that
\[
s_{rn}\ge C s_n, \quad n\in\NN.
\]
\begin{Lemma}[see \cite{AA}*{Lemma 2.12}]\label{lem:RegularLRP}Let $\sss=(s_n)_{n=1}^\infty$ be a essentially increasing weight such that $\ww=(s_n/n)_{n=1}^\infty$ is essentially decreasing. Then $\sss$ has the LRP if and only if $\ww$ is a regular weight.
\end{Lemma}

The following result is rather straightforward, and old-timers will surely be aware of it and could produce its proof on the spot. Nonetheless, 
for later reference and exponential ease, we record it.
\begin{Lemma}\label{lem:symmetricl1}Let $\BB=(\xx_j)_{j=1}^\infty$ be a subsymmetric basis of a Banach space $\XX$ such that
$
n\lesssim \left\Vert \sum_{j=1}^n \xx_j \right\Vert
$
for $n\in\NN$.
Then $\BB$ is equivalent to the unit vector basis of $\ell_1$.
\end{Lemma}

\begin{proof}By \cite{LinTza}*{Proposition 3.a.4}, 
\[
\left|  \sum_{j=1}^n a_j  \right|
\lesssim
\left| \frac{1}{n} \sum_{j=1}^n a_j  \right| \left\Vert  \sum_{j=1}^n \xx_j \right\Vert
\lesssim  \left\Vert f \right\Vert, \quad n\in\NN, \,
f= \sum_{j=1}^\infty a_j \, \xx_j \in\XX.
\]
Combining this inequality (which using the terminology introduced by Singer \cite{Singer3} says that $\BB$ is a basis of type $P^*$) with unconditionality yields the desired result.\end{proof} 

We are ready to state and prove the main theorem of the present section.
\begin{Theorem}\label{thm:LechnerMar}Let $\sss=(s_n)_{n=1}^\infty$ be an increasing weight  whose discrete derivate is essentially decreasing.
Then, the unit vector system of  $m(\sss)$  satisfies Lechner's condition if and only if
\[
S:=\inf_{n\in\NN} \sup_{k\in\NN} \frac{s_k}{s_{kn}}=0.
\]
\end{Theorem} 

\begin{proof}We infer from our assumptions on $\sss$ that there is a non-increasing weight $\ww=(w_n)_{n=1}^\infty$ 
whose primitive weight $\ttt=(t_n)_{n=1}^\infty$ is equivalent to $\sss$. So, $m(\ttt)=m(\sss)$. Moreover, by Lemma~\ref{lem:RegularLRP},  if $\sss$ had the LRP we could choose $\ww$ to be regular.

If $\lim_n w_n>0$ we would have $s_n \approx n$ for $n\in\NN$ and, then,  $S=0$.  We would also have  $m(\sss)=\ell_\infty$. Therefore, by  Proposition~\ref{prop:lechner:3}, $m(\sss)$ would satisfy Lechner's condition. So, we assume from now on that $\ww\in\WW$.  Then, by Proposition~\ref{prop:Mar1}~(i), the unit vector system of $m(\ttt)$ is the dual basic sequence of the unit vector basis of $d(\ww,1)$.

Given a bijection $\pi\colon\NN^2\to\NN$ we define a disjointly supported sequence $\BB_\pi=(f_n)_{n=1}^\infty$ in $\FF^\NN$ by 
\[ 
f_n=(a_{j,n})_{j=1}^\infty, \quad  a_{j,n}=\begin{cases}  w_i & \text{ if } j=\pi(i,n), \\ 0 & \text{ otherwise.}\end{cases} 
\] 
The non-increasing rearrangement of each sequence $f_n$ is the sequence $(w_i)_{i=1}^\infty$. Then, by \eqref{eq:MarNorm}, $\Vert f_n \Vert_{m(\ttt)}=1$. We infer that  $\BB_\pi$ is a symmetric basic sequence in $m(\ttt)$. Given $m\in\NN$ the non-increasing rearrangement of  $\sum_{n=1}^m f_n$ is the sequence 
\[ 
(\underbrace{w_1,\dots,w_1}_{m},\dots, \underbrace{w_k,\dots,w_k}_{m},\dots). 
\] 
Therefore, applying \eqref{eq:MarNorm} and taking into account that $\ww$ is non-in\-creasing,
\[ 
\left\Vert \sum_{n=1}^m f_n \right\Vert_{m(\ttt)} 
=\sup_{\substack{ k \ge 1 \\ 1 \le r \le m}}\frac{ r w_k +m \sum_{i=1}^{k-1}w_i}{\sum_{i=1}^{m(k-1)+r} w_i}\\ 
\le\sup_{\substack{ k \ge 1 \\ 1 \le r \le n}} a_{k,r}^{(m)}, 
\] 
where 
\[  
a_{k,r}^{(m)}=\frac{ r w_k +m \sum_{i=1}^{k-1}w_i}{ \frac{r}{m}\sum_{i=1+m(k-1)}^{mk} w_i+\sum_{i=1}^{m(k-1)} w_i}. 
\] 
Since, for any $a,b,c,d\in(0,\infty)$,  the mapping  $t\mapsto (a+bt)/(c+dt)$ is monotone in $(0,\infty)$ we have 
$a_{k,r}^{(m)}\le\max\{a_{k,0}^{(m)},a_{k,m}^{(m)}\}$ whenever $1\le r \le m$ and $k\ge 2$. Put 
\[ 
b_k^{(m)}=\frac{  \sum_{i=1}^{k}w_i}{ \sum_{i=1}^{mk} w_i}, \, k,m\in\NN,  \quad B^{(m)}=\sup_{k\in\NN} b_k^{(m)}, \, m\in\NN, \quad B=\inf_{m\in\NN} B_m. 
\]  
Note that  $a_{k,m}^{(m)} = a_{k+1,0}^{(m)} = m b_{k}^{(m)}$ for every $k\in\NN$, and that 
$a_{1,r}^{(m)}=m b_1^{(m)}$ for every $r\in\NN$. Consequently, 
\[ 
\left\Vert \sum_{n=1}^m f_n \right\Vert_{m(\ttt)}= m B^{(m)}, \quad m\in\NN.
\]

Taking into account Theorem~\ref{thm:lechner:2}, we have to prove that $B>0$ if and only if $m(\ttt)$ contains a disjointly supported sequence equivalent to the unit vector basis of $\ell_1$.  Assume that  $B>0$. Pick a bijection $\pi$ from $\NN^2$ onto $\NN$ and let $\BB_\pi=(f_n)_{n=1}^\infty$. We have  $B>0$ and $B m \le\Vert \sum_{n=1}^m f_n \Vert_{m(\ttt)}$ for every $m\in\NN$. Then, by Lemma~\ref{lem:symmetricl1}, $\BB_\pi$, regarded as a (disjointly supported) sequence in $m(\ttt)$,  is equivalent to the unit vector system of $\ell_1$. 

Reciprocally, assume that $B=0$. In particular, there is $n\in\NN$ such that $s_k/s_{nk}\le 1/2$ for every $k\in\NN$. Then, $\sss$ has the LRP and,  consequently, we can, and we do, assume that the weight $\ww$ above chosen is regular. Therefore, by Proposition~\ref{prop:Mar1}~(ii), $m(\sss)=m(\ttt)=d(\ww,\infty)$.
Let $(g_n)_{n=1}^\infty$ be a disjointly supported sequence in $\FF^\NN$ with $\sup_n \Vert g_n \Vert_{d(\ww,\infty)}<\infty$. By the very definition of the quasi-norm in $d(\ww,\infty)$, there is a bijection 
$\pi\colon\NN^2\to \NN$ such that $(g_n)_{n=1}^\infty\lesssim \BB_\pi$. Consequently, if $ \BB_\pi=(f_n)_{n=1}^\infty$,
\[
\left\Vert \sum_{n=1}^m g_n \right\Vert_{d(\ww,\infty)}\lesssim  \left\Vert \sum_{n=1}^m f_n \right\Vert_{d(\ww,\infty)}  \le m B^{(m)}, \quad m\in\NN. 
\]
We infer that $\inf_m m^{-1} \Vert \sum_{n=1}^m g_n \Vert_{d(\ww,\infty)}=0$.
Then, $(g_n)_{n=1}^\infty$, regarded as a  sequence in $d(\ww,\infty)$,  is not equivalent to the unit vector system of $\ell_1$.
\end{proof} 

To give relevance to Theorem~\ref{thm:LechnerMar} we make the effort of telling apart Marcinkiewicz spaces from $\ell_\infty$. 
\begin{Proposition}Let $\sss=(s_n)_{n=1}^\infty$ be an increasing weight  whose discrete derivate is essentially decreasing. If
$\lim_n s_n/n=0$ then  $m(\sss)$  is not an $\LL_\infty$-space. In particular, $m(\sss)$ is not isomorphic to $\ell_\infty$.
\end{Proposition}

\begin{proof}
Pick $\ww\in\WW$  whose primitive weight is equivalent to $\sss$. Assume that $m(\sss)$ is an $\LL_\infty$-space. Then,
by  Proposition~\ref{prop:Mar1}~(i) and \cite{LinRos1969}*{Theorem III},  $d(\ww,1)$ is an $\LL_1$-space. Since the unit vector system is an unconditional basis of $d(\ww,1)$, by invoking \cite{LindenstraussPel1968}*{Theorem 6.1}, we reach the absurdity $d(\ww,1)=\ell_1$.
\end{proof}

We close this section by writing down the result that arise from combining Theorem~\ref{thm:LechnerMar} with Theorem~\ref{thm:Lechner}.
\begin{Corollary}Let $p\in[1,\infty]$  and let $\sss=(s_n)_{n=1}^\infty$ be an increasing weight whose discrete derivate is essentially decreasing. Assume that 
\[
\inf_{n\in\NN} \sup_{k\in\NN} \frac{s_k}{s_{kn}}=0.
\]
Let $\YY$ be either $m(\sss)$ or $\ell_p(m(\sss))$. Then, if $T\in\LL(\YY)$, the identity map on  $\YY$  factors through either $T$ or $\Id_\YY-T$. Consequently, $\ell_p(m(\sss))$  is a primary Banach space. 
\end{Corollary} 

\section{Orlicz sequence spaces}\label{Orlicz} 
\noindent 
Throughout this section we follow the terminology on Orlicz spaces and Museilak-Orlicz spaces used in the handbooks \cites{LinTza,Museilak}. A  \textit{normalized convex Orlicz function} is a convex function $M\colon[0,\infty)\to[0,\infty)$ 
such that $M(0)=0$ and $M(1)=1$. If $M$ vanishes in a neighborhood of the origin, $M$ is said to be \textit{degenerate}.
Given a sequence $\MM=(M_n)_{n=1}^\infty$ of normalized convex Orlicz functions, the Museilak-Orlicz norm $\Vert \cdot \Vert_{\ell_\MM}$ is the Luxemburg norm built from the modular 
\[
m_\MM\colon \FF^\NN\to[0,\infty], \quad 
(a_n)_{n=1}^\infty \mapsto \sum_{n=1}^\infty M_n(|a_n|). 
\]
The Museilak-Orlicz space $\ell_\MM$ is the Banach space consisting of all sequences $f$ for which  $\Vert f \Vert_{\ell_\MM}<\infty$.
Orlicz sequence spaces can be obtained as a particular case of  Museilak-Orlicz sequence spaces. Namely, if $M$ is a normalized convex Orlicz functions, we put $\ell_M=\ell_\MM$, where, if  $\MM=(M_n)_{n=1}^\infty$, $M_n=M$ for every $n\in\NN$. We will denote by $h_M$ the closed linear span of the unit vector system of $\ell_M$.
It is known (see \cite{LinTza}*{Proposition 4.a.2}) that 
\begin{equation}\label{eq:separableOrlicz}
h_M=\{ f \in\FF^\NN \colon m_M(sf)<\infty\, \forall s<\infty\}.
\end{equation}

Given a one-to-one map $\phi\colon\NN\to\NN$ we consider the linear operator defined as in \eqref{equation:vf} corresponding to the unit vector system of $\FF^\NN$, that is,
\[
T_\phi\colon\FF^\NN\to \FF^\NN, \quad (a_n)_{n=1}^\infty \mapsto (b_n)_{n=1}^\infty, \quad 
b_n=\begin{cases}a_k &\text{ if } n=\phi(k),\\0 &\text{ otherwise.}\end{cases}
\]
Notice that, if $\MM=(M_n)_{n=1}^\infty$ and $\NNN=(M_{\phi(k)})_{k=1}^\infty$,  $T_\phi$ restricts to an isometric embedding from $\ell_\NNN$ into $\ell_\MM$. 
This claim gives, in particular, that the unit vector system is a symmetric basic sequence in any Orlicz sequence space. Let us bring up the following result that we will need.
\begin{Theorem}[see \cite{Museilak}*{Theorem 8.11}] \label{MuseilakInclusion} Let $\MM=(M_n)_{n=1}^\infty$ and $\NNN=(N_n)_{n=1}^\infty$ be sequences of 
normalized convex Orlicz functions. Then $\ell_\NNN\subseteq \ell_\MM$ if and only if there are a positive sequence $(a_n)_{n=1}^\infty \in \ell_1$, $\delta >0$, $C$ and $D\in(0,\infty)$ such that
\[
N_n(t) < \delta \Longrightarrow M_n(t) \le C N_n(D t) + a_n.
\]
\end{Theorem}
Theorem~\ref{MuseilakInclusion} gives, in particular, that $\ell_M=\ell_N$ if and only if there are $a,b>0$ such that $M(t)\approx N(bt)$ for $0\le t\le a$ (see also \cite{LinTza}*{Proposition 4.a.5}). 

If we denote, for $b\in(0,\infty)$,
\[
M_b(t)=\frac{M(bt)}{M(b)}, \quad t \ge 0,
\]
the indices $\alpha_M$ and $\beta_M$ of the non-degenerate normalized convex Orlicz function $M$ are defined, with the convention $\inf\emptyset=\infty$, by 
\begin{align*}
\alpha_M&=\sup\{q\in[1,\infty) \colon  \sup_{0\le b,t \le 1} t^{-q} M_b(t) <\infty\}, \\
\beta_M&=\inf\{q\in[1,\infty) \colon  \inf_{0\le b,t \le 1} t^{-q} M_b(t)>0\}.
\end{align*}
By convexity, $M(bt)\le t M(b)$ for every $t\in[0,1]$ and $b\in[0,\infty)$. If $q\in[1,\infty]$ is such that $\sup\{ t^{-q} M_b(t)\colon 0< b,t \le 1\}  <\infty$ and 
$ \inf\{ t^{-q} M_b(t) \colon 0< b,t \le 1\} >0$, then
$\sup\{ t^{-r} M_b(t)\colon 0< b,t \le 1\}  =\infty$ for every $q<r$ and $\inf\{ t^{-s} M_b(t)\colon 0< b,t \le 1\} =0$ for every $s<q$.
Consequently, $1\le \alpha_M\le \beta_M\le\infty$. Our characterization of Orlicz sequence spaces satisfying Lechner's condition will be a consequence following result.

\begin{Theorem}[cf. \cite{LinTza}*{Theorem 4.a.9}]\label{thm:Orlicz:1}Let $M$ be a non-degenerate normalized convex Orlicz function and
$1\le p\le\infty$.  The following are equivalent.
\begin{itemize}
\item[(a)] $p\in[\alpha_M,\beta_M]$.
\item[(b)] There is a disjointly supported sequence, with respect to the unit vector system of $\ell_M$, that is equivalent to the unit vector system of $\ell_p$.
\item[(c)] There is a  block basic sequence with respect to the unit vector system of $\ell_M$ that is equivalent to the unit vector system of $\ell_p$.
\end{itemize}
\end{Theorem}
We emphasize that the equivalence between  items (a) and (c) can be easily obtained from \cite{LinTza}*{Theorem 4.a.9}. Indeed, it follows from combining
\cite{AAW}*{Proposition 2.14},  \cite{AK}*{Theorem 3.3.1} and Bessaga-Pelczy\'nski Selection Principle that if a Banach space  provided with an unconditional basis $\UU$ contains a subsymmetric basic sequence $\BB$, then there is a block basic sequence with respect to $\UU$ that is equivalent to $\BB$.
Since it is obvious that (c) implies (b), our contribution to the theory of sequence Orlicz spaces consists in proving that (b) implies (a).  Nonetheless, for expository ease, we will put in order all the arguments that come into play in the proof of Theorem~\ref{thm:Orlicz:1}. We start by writing down some  terminology and claims from \cite{LinTza}.

Given a non-degenerate normalized convex Orlicz function, the set $C_{M,1}\subseteq \Cont([0,1/2])$ is  the smallest closed convex set containing $\{M_b \colon 0<b\le 1\}$. Note that every function in $C_{M,1}$ extends to a normalized convex Orlicz function. So, we can safely define $\ell_F$ for $F\in C_{M,1}$.

\begin{Theorem}[cf. \cite{LinTza}*{Theorem 4.a.8}]\label{thm:Orlicz:2}
Let $M$ and $F$ be  normalized convex Orlicz functions.  Assume that $M$ is non-degenerate and that $F\in C_{M,1}$. Then, there is a block basic sequence of the unit vector system of $\ell_M$ that is equivalent to the unit vector system of $\ell_F$.
\end{Theorem}
\begin{proof}Lindenstrauss-Tzafriri's  proof of the ``if''  part of \cite{LinTza}*{Theorem 4.a.8}   gives exactly this result.
\end{proof} 

Given $p\in[1,\infty)$, $F_p$ will denote the potential funtion given by $F_p(t)=t^p$, $t\ge 0$. We denote by 
$M_\infty$ the degenerate Orlicz function defined  by  $M(t)=0$ if $0\le t \le 1/2$ and $M(t)=2t-1$ if $t>1/2$. Of course, $\ell_{F_p}=h_{F_p}=\ell_p$ for $1\le p<\infty$, $\ell_{F_\infty}=\ell_\infty$, and $h_{F_\infty}=c_0$.
\begin{Theorem}[see \cite{LinTza}*{Comments below Theorem 4.a.9}]\label{thm:Orlicz:3} Let $M$  be a non-degenerate normalized convex Orlicz function and $p\in[\alpha_M,\beta_M]$. Then $F_p\in C_{M,1}$.
\end{Theorem}
\begin{proof}Lindenstrauss-Tzafriri's  proof of the ``if''  part of  \cite{LinTza}*{Theorem 4.a.9} contains a proof of this result.
\end{proof}
We say that a function $M\colon[0,\infty)\to[0,\infty)$ satisfies the $\Delta_2$-condition at zero if  there is $a>0$ such that $M(2t)\lesssim M(t)$ for $0\le t \le a$.
Note that a non-degenerate normalized convex Orlicz function $M$ satisfies the $\Delta_2$-condition at zero if and only if $M(t)\lesssim M(t/2)$ for $0\le t \le 1$.

\begin{Theorem}[cf. \cite{LinTza}*{Proof of Theorem 4.a.9}]\label{thm:Orlicz:4}
Let $M$ be a non-degenerate normalized convex Orlicz function. The following are equivalent.

\begin{itemize}
\item[(a)] $\beta_M<\infty$.

\item[(b)] $M$ satisfies the $\Delta_2$-condition at zero.

\item[(c)] $\ell_M=h_M$.

\item[(d)]$F_\infty\notin C_{M,1}$.
\end{itemize}
\end{Theorem}

\begin{proof}First, we prove (a) $\Longrightarrow$ (b). Assume that  $\beta_M<\infty$. Then there are $1\le q ,C<\infty$ such that  $M(b)\le C t^{-q} M(bt)$ for every $(b,t)\in(0,1]^2$. In particular,
$
M(b)\le C 2^q M(b/2)
$
for every $0<b\le 1$.

(b) $\Longrightarrow$ (c) is a part of  \cite{LinTza}*{Theorem 4.a.4}.
Second, we prove (c) $\Longrightarrow$ (d). If  $\ell_M=h_M$, then the unit vector system is a boundedly complete basis of $\ell_M$. Then, by \cite{AK}*{Theorem 3.3.2}, no basic sequence of the unit vector system of $\ell_M$ is equivalent to the unit vector system of $\ell_\infty=\ell_{F_\infty}$. By Theorem~\ref{thm:Orlicz:2}, $F_\infty\notin C_{M,1}$.

Third,  we prove (d) $\Longrightarrow$ (b). Assume that $F_\infty\notin C_{M,1}$. Then there is constant $c>0$ such that $\sup\{ M_b(t) \colon 0\le t \le 1/2\}=M_b(1/2)\ge c$ for every $b\in(0,1]$. In other words, $M(b)\le c^{-1} M(b/2)$ for every $0\le b \le 1$.

Finally, we prove (b) $\Longrightarrow$ (a). Let $C\ge 2$ be such that $M(b)\le C M(b/2)$ for every $b\in(0,1]$.  Choose $q=\log_2(C)$. Given $0<t\le 1$, pick $n\in\NN$
such that $2^{-n}<t\le 2^{-n+1}$. We have
\[
M(b)\le C^n M(2^{-n} b)=C 2^{(n-1)q}  M(2^{-n} b)\le C t^{-q} M(tb).
\]
Therefore, $\inf \{ t^{-q} M_b(t) \colon 0<b\le 1\}\ge C^{-1}>0$. Consequently, $\beta_M\le q <\infty$.
\end{proof}

In our route to prove Theorem~\ref{thm:Orlicz:1} we  need to study functions constructed from sequences belonging to Orlicz spaces. Given a normalized convex Orlicz function $M$
and $f=(b_j)_{j=1}^\infty\in\FF^\NN$ we define
\[
M_f\colon[0,\infty)\to[0,\infty], \quad s\mapsto \sum_{j=1}^\infty M(|b_j| s).
\]

\begin{Lemma}\label{lem:Orlicz:5}Let $M$ be a normalized convex Orlicz function and $f\in\FF^\NN$ with $0<R:=\Vert f \Vert_{\ell_M}<\infty$.
\begin{itemize}
\item[(a)] $\{ s \in [0,\infty] \colon M_f(s)\le 1\}=[0,1/R]$.
\item[(b)] If there is $s>1/R$ such that $M_f(s)<\infty$, then $M_f(1/R)=1$.
\end{itemize}
\end{Lemma}

\begin{proof}By definition, $M_f(s)\le 1$ if $s<1/R$ and  $M_f(s)> 1$ if $s>1/R$. By  the monotone convergence theorem, $M_f(1/R)\le 1$. By the dominated convergence theorem,
$M_f$ is continuous in the interval $\{ s\in[0,\infty) \colon M_f(s)<\infty\}$. Then, (b) holds.
\end{proof}

\begin{Proposition}\label{prop:Orlicz:6}Let $M$ be a non-degenerate normalized convex Orlicz function and $f\in S_{\ell_M}$. Then $M_f\in C_{M,1}$.
\end{Proposition}

\begin{proof}Assume, without loss of generality, that $f=(b_j)_{j=1}^\infty\in[0,\infty)^\NN$. Denote $\lambda_j=M(b_j)$ for $j\in B:=\supp(f)$. By Lemma~\ref{lem:Orlicz:5},
$\lambda_\infty:=1-\sum_{j\in B} \lambda_j\ge 0$. If $\lambda_\infty=0$, 
\[
M_f=\sum_{j\in B} \lambda_j M_{b_j} \text{ and } \sum_{j\in B} \lambda_j=1.
\]
If $\lambda_\infty>0$, then, by Lemma~\ref{lem:Orlicz:5} and the identity \eqref{eq:separableOrlicz}, $f\in\ell_M\setminus h_M$. Therefore, by Theorem~\ref{thm:Orlicz:4}, 
$F_\infty\in C_{M,1}$. We would have
\[
M_f=\lambda_\infty F_\infty + \sum_{j\in B} \lambda_j M_{b_j} \text{ and } \lambda_\infty + \sum_{j\in B} \lambda_j=1.
\]
In both cases, we obtain that $M_f$ is a (possibly infinite) convex combination of functions in $C_{M,1}$. Consequently, $M_f\in C_{M,1}$.
\end{proof}

We are now in a position to complete the proof of Theorem~\ref{thm:Orlicz:1}.

\begin{proof}[Proof of Theorem~\ref{thm:Orlicz:1}]
The proof we present here is inspired by that of  \cite{LinTza}*{Theorem 4.a.9}. 
(a) $\Longrightarrow$ (c) follows from combining Theorem~\ref{thm:Orlicz:3} with Theorem~\ref{thm:Orlicz:2}, and (c) $\Longrightarrow$ (b) is obvious.
In order to prove that (b) $\Longrightarrow$ (a), let $(f_n)_{n=1}^\infty$  be a disjointly supported sequence in $\ell_M$  that is equivalent to the unit vector basis of $\ell_p$.
By unconditionality, we can assume, without loss of generality, that  $\Vert f_n\Vert_{\ell_M}=1$ for every $n\in\NN$. Then, by Proposition~\ref{prop:Orlicz:6},
$M_{f_n}\in C_{M,1}$ for every $n\in\NN$.
By \cite{LinTza}*{Lemma 4.a.6},  $C_{M,1}$ is compact and, then, there is $F\in C_{M,1}$ such that 
\begin{equation}\label{eq:1}
\sup_{0\le t \le 1/2} |M_{f_{\phi(k)}}(t) -F(t)|\le 2^{-k}
\end{equation}
for some increasing map $\phi\colon\NN\to\NN$. Then, if we put $\NNN=(M_{f_{\phi(k)}})_{k=1}^\infty$, applying Theorem~\ref{MuseilakInclusion}  yields $\ell_\NNN=\ell_F$.

Let $f_n=(b_{j,n})_{j=1}^\infty$ for every $n\in\NN$. For any $(a_n)_{n=1}^\infty\in\FF^\NN$ we have
\[
\left\Vert \sum_{n=1}^\infty a_n \, f_n\right\Vert_{\ell_M}=\inf\left\{ t \colon \sum_{n=1}^\infty \sum_{j=1}^{\infty} M \left(\frac{| a_n b_{j,n}|}{ t}\right)\le 1\right\}
=\left\Vert (a_n)_{n=1}^\infty\right\Vert_{\ell_\MM}.
\]
Consequently, for $g=(a_k)_{k=1}^\infty\in\FF^\NN$,
\[
\left\Vert \sum_{k=1}^\infty a_k \, f_{n_k}\right\Vert_{\ell_M}
=\left\Vert T_\phi\left( g \right)\right\Vert_{\ell_\MM}
=\Vert g \Vert_{\ell_\NNN}
\approx \Vert  g \Vert_{\ell_F}.
\]
Since  $(f_{n_k})_{k=1}^\infty$ is equivalent to the unit vector basis of $\ell_p$,  we obtain $\ell_F=\ell_p$. Therefore, there is $a>0$ such that $F(t)\approx F_p(t)$ for $0\le t\le a$. 

Let $r<\alpha_M$. There is a constant $C_1<\infty$ such that $ M_b(t) \le C_1 t^r$ for every $0<b\le 1$ and every $0\le t \le 1$. By convexity and continuity, $N(t) \le C_1 t^r$ for every $N\in C_{M,1}$ and  every $0\le t \le1/2$. Consequently, there is $C_2<\infty$ such that $F_p(t)\le C_2 t^r$ for every $0\le t\le a$. We infer that $r\le p$.  Letting $r$ tend to $\alpha_M$ we obtain $\alpha_M\le p$. We prove that $p\le \beta_M$ in a similar way.
\end{proof}

\begin{Theorem}\label{LechnerOrlicz}Let $M$ be a non-degenerate normalized convex Orlicz  function. Then the unit vector system of $\ell_M$ verifies Lechner's condition if and only if $1<\alpha_M$. 
\end{Theorem}

\begin{proof}Let $M^*$ be the complementary Orlicz function of $M$. By \cite{LinTza}*{Proposition 4.b.1}, the unit vector system of 
$\ell_M$ is, under the natural pairing, the dual basic sequence of the unit vector basis of  $h_{M^*}$. Then, the result follows from combining Theorem~\ref{thm:lechner:2}, with Theorem~\ref{thm:Orlicz:1}.\end{proof}   

As in Section~\ref{Marcin}, we close by telling apart Orlicz sequence spaces from $\ell_\infty$ and writing down the straightforward consequence of combining Theorem~\ref{thm:Lechner} with Corollary~\ref{LechnerOrlicz}. We emphasize that, in light of Theorem~\ref{thm:CKLSC} and the results achieved in \cite{CL}, Theorem~\ref{thm:LechnerOrlicz} is a novelty only in the case when $\ell_M$ is not separable, i.e., when $\beta_M=\infty$. 

\begin{Proposition}Let $M$ be a normalized convex Orlicz function. Then $\ell_M$ is a $\LL_\infty$-space if and only if $M$ is degenerate.
\end{Proposition}

\begin{proof}Let $M^*$ be the complementary Orlicz function of $M$. If $\ell_M$ is a $\LL_\infty$-space, then $h_{M^*}$ is a $\LL_1$-space.
Since the unit vector basis of $h_{M^*}$ is unconditional, we obtain $h_{M^*}=\ell_1$. Therefore, $\ell_M=\ell_\infty=\ell_{F_\infty}$. 
Consequently, there is $0<a\le 1/2$ such that $M(t)=F_\infty(t)=0$ for every $0\le t \le a$.
\end{proof} 

\begin{Theorem}\label{thm:LechnerOrlicz} Let $p\in[1,\infty]$ and $M$ be a non-degenerate normalized convex Orlicz  function with $\alpha_M>1$. 
Let $\YY=\ell_M$ or $\YY=\ell_p(\ell_M)$. If $T\in\LL(\YY)$, then the identity map on  $\YY$ factors through either $T$ or $\Id_\YY-T$.  
Consequently,  $\ell_p(\ell_M)$  is a primary Banach space. 
\end{Theorem}


\begin{bibsection}
\begin{biblist}

\bib{AA}{article}{   
author={Albiac, F.},   
author={Ansorena, J.~L.},   
title={Lorentz spaces and embeddings induced by almost greedy bases in   
Banach spaces},   
journal={Constr. Approx.},   
volume={43},   
date={2016},   
number={2},   
pages={197--215}, 
}

\bib{AADK}{article}{   
author={Albiac, F.},   
author={Ansorena, J.~L.},   
author={Dilworth, S.~J.},   
author={Kutzarova, Denka},   
title={Banach spaces with a unique greedy basis},   
journal={J. Approx. Theory},   
volume={210},   
date={2016},   
pages={80--102}, 
}

\bib{AAW}{article}{
author={Albiac, F.},
author={Ansorena, J.~L.},
author={Wallis, B.},
title={Garling sequence spaces},
journal={J. London Math. Soc.},
volume={98},
 date={2018},
 number={2},
 pages={204--222}, 
 }


\bib{AK}{book}{   
author={Albiac, F.},   
author={Kalton, N.~J.},   
title={Topics in Banach space theory},   
series={Graduate Texts in Mathematics},   
volume={233},   
edition={2},   
publisher={Springer, [Cham]},   
date={2016},  
}

\bib{Ansorena}{article}{   
author={Ansorena, J.~L.},   
title={A note on subsymmetric renormings of Banach spaces},   
journal={Quaest. Math.},   
volume={41},   
date={2018},   
number={5},   
pages={615--628},  
}

\bib{Capon}{article}{   
author={Capon, M.},   
title={Primarit\'e de certains espaces de Banach},   
language={French},   
journal={Proc. London Math. Soc. (3)},   
volume={45},   
date={1982},   
number={1},   
pages={113--130}, 
}

\bib{CL}{article}{   
author={Casazza, P. G.},   
author={Lin, B. L.},   
title={Projections on Banach spaces with symmetric bases},   
journal={Studia Math.},   
volume={52},   
date={1974},   
pages={189--193},   
}

\bib{CKL}{article}{   
author={Casazza, P. G.},   
author={Kottman, C. A.},   
author={Lin, B. L.},   
title={On some classes of primary Banach spaces},   
journal={Canad. J. Math.},   
volume={29},   
date={1977},   
number={4},   
pages={856--873}, 
}

\bib{CRS}{article}{   
author={Carro, M. J.},   
author={Raposo, J. A.},   
author={Soria, J.},   
title={Recent developments in the theory of Lorentz spaces and weighted   
inequalities},   
journal={Mem. Amer. Math. Soc.},   
volume={187},   
date={2007},   
number={877},   
pages={xii+128}, 
}

\bib{James1950}{article}{   
author={James, R.~C.},   
title={Bases and reflexivity of Banach spaces},   
journal={Ann. of Math. (2)},   
volume={52},   
date={1950},   
pages={518--527},   
}

\bib{KP}{article}{   
author={Kadec, M. I.},   
author={Pe\l czy\'nski, A.},   
title={Bases, lacunary sequences and complemented subspaces in the spaces   
$L_{p}$},   
journal={Studia Math.},   
volume={21},   
date={1961/1962},   
pages={161--176}, 
}

\bib{Lechner}{article}{   
author={Lechner, R.},   
title={Subsymmetric weak* Schauder bases and factorization of the identity},   
journal={arXiv:1804.01372 [math.FA]},   
date={2018},
}

\bib{LindenstraussPel1968}{article}{   
author={Lindenstrauss, J.},   
author={Pe{\l}czy{\'n}ski, A.},   
title={Absolutely summing operators in $L_{p}$-spaces and their   
applications},   
journal={Studia Math.},   
volume={29},   
date={1968},   
pages={275--326}, 
}

\bib{LinRos1969}{article}{   
author={Lindenstrauss, J.},   
author={Rosenthal, H.P.},   
title={The $\LL_p$ spaces},   
journal={Israel J. Math.},   
volume={7},   
date={1969},   
pages={325--349},  
}

\bib{LinTza}{book}{  
author={Lindenstrauss, J.},  
author={Tzafriri, L.},  
title={Classical Banach spaces. I},  
note={Sequence spaces;  
Ergebnisse der Mathematik und ihrer Grenzgebiete, Vol. 92},  
publisher={Springer-Verlag, Berlin-New York},  
date={1977},  
pages={xiii+188},
}

\bib{Museilak}{book}{   
author={Musielak, J.},   
title={Orlicz spaces and modular spaces},   
series={Lecture Notes in Mathematics},   
volume={1034},   
publisher={Springer-Verlag, Berlin},   
date={1983},   
pages={iii+222}, 
}

\bib{Samuel}{article}{   
author={Samuel, C.},   
title={Primarit\'e des produits d'espaces de suites},   
language={French},   
journal={Colloq. Math.},   
volume={39},   
date={1978},   
number={1},   
pages={123--132},
}

\bib{Singer1}{article}{   
author={Singer, I.},   
title={On Banach spaces with symmetric basis},   
language={Russian},   
journal={Rev. Math. Pures Appl.},   
volume={6},   
date={1961},   
pages={159--166},   
}

\bib{Singer2}{article}{   
author={Singer, I.},   
title={Some characterizations of symmetric bases in Banach spaces},   
journal={Bull. Acad. Polon. Sci. S\'er. Sci. Math. Astronom. Phys.},   
volume={10},   
date={1962},   
pages={185--192},   
}

\bib{Singer3}{article}{   
author={Singer, I.},   
title={Basic sequences and reflexivity of Banach spaces},   
journal={Studia Math.},   
volume={21},   
date={1961/1962},   
pages={351--369},
}

\end{biblist}

\end{bibsection}
\end{document}